\documentclass[11pt,letterpaper,reqno]{amsart}
\usepackage{graphicx}
\usepackage{amsmath}
\usepackage{amssymb}
\usepackage{amsfonts}
\usepackage{amsrefs}
\usepackage{enumerate}
\usepackage{cancel}
\usepackage[mathscr]{eucal}

\usepackage[normalem]{ulem}

\usepackage[usenames]{color}

\newcommand{\R}{\mathbb{R}}

\newcommand{\I}{{\rm{\mathbf{I}}}}
\newcommand{\II}{{\rm{\mathbf{II}}}}

\newtheorem{thm}{Theorem}
\newtheorem{lemma}[thm]{Lemma}
\newtheorem{prop}[thm]{Proposition}
\newtheorem{cor}[thm]{Corollary}
\newtheorem{conj}[thm]{Conjecture}

\theoremstyle{definition}

\theoremstyle{remark}
\newtheorem{remark}{Remark}

\newtheorem*{outline}{Outline}

\DeclareMathOperator{\Hess}{Hess}
\DeclareMathOperator{\capac}{cap}

\DeclareMathOperator{\spt}{spt}

\begin{document}

\title{Scalar curvature and the relative capacity of geodesic balls}
\author{Jeffrey L. Jauregui}
\date{\today}

\begin{abstract}
In a Riemannian manifold, it is well known that the scalar curvature at a point can be recovered from the volumes (areas) of small geodesic balls (spheres). We show the scalar curvature is likewise determined by the relative capacities of concentric small geodesic balls. This result has motivation from general relativity (as a complement to a previous study by the author of the capacity of large balls in an asymptotically flat manifold) and from weak definitions of nonnegative scalar curvature. It also motivates a conjecture (inspired by the famous volume conjecture of Gray and Vanhecke), regarding whether Euclidean-like behavior of the relative capacity on the small scale is sufficient to characterize a space as flat.
\end{abstract}

\maketitle 

\section{Introduction}
Let $(M,g)$ be a Riemannian $n$-manifold, $n \geq 3$, and let $p \in M$. Let $V(r)$ and $A(r)$ be the volume and boundary hypersurface area of the geodesic ball of radius $r$ about $p$ with respect to $g$. The following expansions for small $r$ are well known (see \cite{GV} for instance):
\begin{align}
V(r) &= \beta_n r^n\left( 1 - \frac{S(p)}{6(n+2)}r^2 + O(r^4)\right),\label{V}\\
A(r) &= \omega_{n-1} r^{n-1}\left(1-\frac{S(p)}{6n}r^2 + O(r^4)\right),\label{A}
\end{align}
where $\beta_n$ and $\omega_{n-1}=n\beta_n$  are the volume and hypersurface boundary area of the unit $n$-ball in $\R^n$ and $S(p)$ is the scalar curvature at $p$. In particular, $S(p)$ can be detected from the first nontrivial, non-Euclidean term in the expansions for volume and area. 

In this paper we ask: can the scalar curvature be detected from the \emph{capacity} of small geodesic balls?
Since the capacity is a global concept (in contrast to volume and area), we immediately refine the question by localizing and considering instead the capacity of a ball of radius $R_1$ \emph{relative} to a concentric ball of radius $R_2> R_1$ with $R_2$ small. (Capacity and relative capacity are recalled below.) The less explicit nature of the definition of relative capacity and the dependence on two parameters makes this question more subtle for relative capacity than for volume or area.

One source of motivation for the above question lies in trying to understand scalar curvature without relying on regularity of the underlying metric. There has been considerable interest in weak definitions of scalar curvature, particularly lower bounds, with multiple approaches taken (for example, see \cites{Bur,CM, Gro,Lee} and the references therein). Since the definition of capacity does not use derivatives of the Riemannian metric, a new weak definition of nonnegative scalar curvature will be suggested by the main theorem --- we discuss this further immediately after Corollary \ref{cor_main}.

Another source of motivation, this one from general relativity, is the recent approach by the author to detecting the total mass of an asymptotically flat 3-manifold based on the capacity-volume relationship of large regions \cite{Jau}. This was inspired by Huisken's definition of isoperimetric mass \cite{Hui}.
Several known ``quasi-local mass'' quantities, including the Hawking mass, the Brown--York mass, and Huisken's isoperimetric mass detect the scalar curvature (which represents energy density) on the small scale and the total mass on the large scale (see \cite{FST}). Since capacity was used to study total mass on the large scale in \cite{Jau}, it is natural ask the complementary question regarding whether it also detects scalar curvature on the small scale.

Secondarily, we are motivated by extending the known connections between scalar curvature and capacity/harmonic functions (as we recall below, harmonic functions determine the capacity). We mention some of these here (but do not attempt to give a complete list).  Sets of zero capacity play a significant role in Schoen and Yau's study of domains in $S^n$ equipped with conformal metrics with scalar curvature bounds \cite{SY}.
 Bray  \cite{Bray_RPI}, Bray and Miao \cite{BM}, Schwartz \cite{Schw}, Freire and Schwartz \cite{FS}, and Mantoulidis, Miao, and Tam \cite{MMT} have proved inequalities for the boundary capacity of asymptotically flat manifolds that rely on nonnegative scalar curvature. Stern proved a formula relating scalar curvature to the level sets of harmonic functions \cite{Ste} that has found applications including a new proof of the positive mass theorem in dimension three, due to Bray, Kazaras, Khuri, and Stern \cite{BKKS}. Harmonic functions also play an important role in the behavior of scalar curvature under conformal transformations: when used as conformal factors, they preserve the pointwise sign of scalar curvature.

\medskip

Before stating the main result, we recall the definition of relative capacity. Let $\Omega$ be an open subset of a Riemannian $n$-manifold $(M,g)$, $n \geq 3$, and let $K \subset \Omega$ be a compact set. The \emph{relative capacity} of $K$ in $\Omega$ is defined as
\begin{equation}
\label{eqn_cap}
\capac_g(K; \Omega) = \inf_{\phi} \left\{\frac{1}{(n-2)\omega_{n-1}}  \int_{M} |\nabla \phi|^2 dV\; :\; \phi \text{ is Lipschitz}, \phi \equiv 0 \text{ on } K, \phi \equiv 1 \text{ on } M \setminus \Omega\right\},
\end{equation}
where the gradient norm and volume form are with respect to $g$. (If $\Omega$ has non-compact closure, it is also required that $\spt(1-\phi)$ is compact.) In the case $\Omega = M$, then the above is simply the \emph{capacity} of $K$. If $\partial K$ and $\partial \Omega$ are smooth and nonempty and $\Omega$ has compact closure, then there exists a minimizer, namely the unique $g$-harmonic function $u$ on the closure of $\Omega \setminus K$, with $u|_{\partial K} = 0$, $u|_{\partial \Omega}=1$. Thus,
$$\capac_g(K; \Omega) = \frac{1}{(n-2)\omega_{n-1}}  \int_{M} |\nabla u|^2 dV = \frac{1}{(n-2)\omega_{n-1}}  \int_{\partial K} \frac{\partial u}{\partial \nu} dA,$$
where $\nu$ is the $g$-unit normal to $\partial K$ pointing out of $K$. From this it is easy to verify that in Euclidean $n$-space, the relative capacity of concentric balls of radii $R_1 < R_2$ is given by:
\begin{equation}
\label{eqn_euc_cap}
c_n(R_1,R_2) := \frac{1}{(R_1)^{2-n} - (R_2)^{2-n}}.
\end{equation}
Taking the limit $R_2 \to \infty$ recovers the capacity of a ball of radius $R_1$, i.e., the value $(R_1)^{n-2}$. Alternatively, fixing $R_2>0$ and letting $R_1 \to 0$ gives a value asymptotic to $(R_1)^{n-2}$.  If $R_1 \nearrow R_2$, then $c_n(R_1,R_2)$ blows up to $+\infty$. Below we will fix the ratio $\frac{R_2}{R_1}$ as a constant, $\lambda$, and consider the single parameter $r=R_1$. Note $c_n(r,\lambda r)$ is then simply the polynomial $\frac{r^{n-2}}{1-\lambda^{2-n}}$.

\medskip

Our main result is an expansion for the relative capacity of small balls in a Riemannian manifold, giving the next nontrivial term after the leading Euclidean term and showing in particular it is determined by the scalar curvature.
$\;$\\
$\;$\\
$\;$\\
$\;$\\

\begin{thm}
\label{thm_main}
Let $(M,g)$ be a Riemannian $n$-manifold, $n \geq 3$. Fix $p \in M$ and a parameter $\lambda > 1$. If $\capac_g(r,\lambda r)$ denotes the  capacity of the geodesic ball of radius $r$ about $p$ relative to the geodesic ball of radius $\lambda r$ about $p$, then for $r>0$ small,
\begin{equation}
\label{eqn_main0}
\capac_g(r,\lambda r) =\begin{cases}
\frac{r}{1-\lambda^{-1}}\left(1- \frac{ S(p) }{18} \lambda r^2 + o(r^2)\right), & n=3\\
\frac{r^2}{1-\lambda^{-2}} \left(1-\frac{S(p)}{12}\frac{\log(\lambda)}{(1-\lambda^{-2})}r^2+o(r^2)\right), &n=4\\
\frac{r^{n-2}}{1-\lambda^{2-n}}\left(1-\frac{(n-2)S(p)}{6n(n-4)}\frac{1-\lambda^{4-n}}{1-\lambda^{2-n}}r^2 + o(r^2)\right), &n\geq 5.
\end{cases}
\end{equation}
This can be represented with the unified formula for all $n \geq 3$:
\begin{equation}
\label{eqn_main}
\capac_g(r,\lambda r) = c_n(r,\lambda r) \left(1-\frac{(n-2)S(p)}{6n|n-4|^*}  \cdot \frac{c_n(r,\lambda r)}{c_{n-2}(r,\lambda r)} + o(r^2)\right),
\end{equation}
where the $*$ indicates that the $|n-4|$ factor is omitted for $n=4$.
\end{thm}

We briefly explain how \eqref{eqn_main} follows from \eqref{eqn_main0}. We note that $\frac{c_n(r,\lambda r)}{c_{n-2}(r, \lambda r)} = \frac{1-\lambda^{4-n}}{1-\lambda^{2-n}} r^2$ for $n\geq 5$, the term appearing in \eqref{eqn_main0}. For the low dimensions we define\footnote{These expressions have natural interpretations as the one- and two-dimensional relative capacities of a ball in $\R$ or $\R^2$ of radius $r$ relative to a concentric ball of radius $\lambda r$, provided we define the capacity in one and two dimensions as in \eqref{eqn_cap} (with normalization factors of $\frac{1}{2}$ for $n=1$ and $\frac{1}{2\pi}$ for $n=2$.)}:
$$c_1(r, \lambda r) = \frac{1}{(\lambda-1)r}, \qquad c_2(r,\lambda r) = \frac{1}{\log \lambda},$$
and it is then easy to see \eqref{eqn_main} agrees with \eqref{eqn_main0} in all cases. (Thus, the formula for $\capac_g(r,\lambda r)$ involves the Euclidean relative capacity in dimensions $n$ and $n-2$ --- this explains the anomaly in dimension four, as harmonic functions behave differently in dimension two.)

Theorem \ref{thm_main} shows that scalar curvature is determined by the local behavior of capacity, or equivalently, by the local behavior of  the ``nearly radial'' harmonic functions in Lemma \ref{lemma_grad_u}.

The following corollary is immediate: nonnegative scalar curvature essentially corresponds to the relative capacity being smaller in the manifold than in Euclidean space.
\begin{cor}
\label{cor_main}
Let $(M,g)$ be a Riemannian $n$-manifold, $n \geq 3$. Then the scalar curvature at $p \in M$ is nonnegative if and only if
\begin{equation}
\label{eqn_nnsc}
\lim_{r \to 0} \frac{1}{r^2} \left(1- \frac{\capac_g (r,\lambda r)}{c_n(r,\lambda r)}\right) \geq 0.
\end{equation}
\end{cor}

We note this suggests a possible definition for weakly nonnegative scalar curvature in any metric space in which the capacity can be defined (in the same way that \eqref{V} is well-known to lead to a definition of nonnegative scalar curvature in metric spaces in which volume is defined). This includes, for example, $C^0$ Riemannian manifolds. Specifically, one could say the scalar curvature at a point $p$ is weakly nonnegative if \eqref{eqn_nnsc} holds with the limit replaced with a $\liminf$, for all (or, perhaps, some) $\lambda > 1$. 

\medskip

It would be interesting to compute higher-order terms in the expansion for $\capac_g(r,\lambda r)$ for $r$ small. This is more difficult than the corresponding problem for volumes and areas of small balls, where more terms are known  \cite{GV}. The same method in section \ref{sec_ub} for finding an upper bound could easily be extended using the additional known terms in the expansion for $A(r)$, though it is not clear such a bound would be sharp at the $r^{n+2}$ level. A better lower bound would be more challenging to determine, as the techniques in section \ref{sec_lb} do not seem to readily yield additional information.

One source of interest in higher-order terms for $V(r)$ came from the famous conjecture of Gray and Vanhecke \cite{GV} that if in a Riemannian manifold $V(r)= \beta_n r^n$ at every point for small $r$, then the manifold is flat. Based on this, we conjecture the following (for which knowledge of the higher-order terms would be very useful):
\begin{conj}
\label{conj_cap}
Let $(M,g)$ be a Riemannian manifold of dimension $n \geq 3$. If for every $p \in M$, the relative capacity $\capac_g(r,\lambda r)$ equals $\frac{r^{n-2}}{1-\lambda^{2-n}}$ for some $\lambda > 1$, then $g$ is flat.
\end{conj}
Theorem \ref{thm_main} gives partial progress: it implies that such $g$ is scalar-flat.

\begin{outline}
In sections \ref{sec_ub} and \ref{sec_lb} below we establish upper and lower bounds for the relative capacity in Propositions \ref{prop_ub} and \ref{prop_lb}. Together these will immediately imply Theorem \ref{thm_main}.
\end{outline}

\section{Upper bound on capacity}
\label{sec_ub}
In this section we find an upper bound on the relative capacity of small concentric geodesic balls. Throughout the paper, we will use $B_g(p,r)$ to denote the open geodesic ball of radius $r$ with respect to $g$ about $p$; $\capac_g(r, \lambda r)$  will denote the capacity of $\overline{B_g}(p,r)$ relative to $B_g(p, \lambda r)$.

\begin{prop}
\label{prop_ub}
Let $(M,g)$ be a Riemannian manifold of dimension $n \geq 3$, and let $p \in M$. Fix $\lambda > 1$. Then for $r>0$ small:
$$\capac_g(r, \lambda r)  \leq \begin{cases}
\frac{r}{1-\lambda^{-1}}\left(1- \frac{ S(p) }{18} \lambda r^2 + O(r^4)\right), & n=3\\
\frac{r^2}{1-\lambda^{-2}} \left(1-\frac{S(p)}{12}\frac{\log(\lambda)}{(1-\lambda^{-2})}r^2+O(r^4)\right), &n=4\\
\frac{r^{n-2}}{1-\lambda^{2-n}}\left(1-\frac{(n-2)S(p)}{6n(n-4)}\frac{1-\lambda^{4-n}}{1-\lambda^{2-n}}r^2 + O(r^4)\right), &n\geq 5.
\end{cases}$$
\end{prop}

\begin{proof}
We first recall a classical technique for estimating the capacity from above, described in \cite[Section 2.5]{PS}. The idea is to prescribe the level sets of a function, then consider test functions that are constant on these level sets. In this way one is essentially considering a function of a single variable; by making a good choice of such function, an upper bound follows. We recall the details below, based on the exposition in \cite{BM}.

Let $(M,g)$ be a Riemannian $n$-manifold, $n \geq 3$.
Suppose $K \subset \Omega \subseteq M$, where $K$ is  nonempty and compact and $\Omega$ is open with compact closure. Fix $\psi$ as a smooth function on the closure of $\Omega \setminus K$ that vanishes on $\partial K$ and equals a constant $c>0$ on $\partial \Omega$, with $0 \leq \psi \leq c$. Consider a Lipschitz function $f:[0,c] \to \R$ with $f(0)=0$ and $f(c)=1$, and let $\phi = f \circ \psi$. Then using the co-area formula,
\begin{align*}
(n-2)\omega_{n-1} \capac_g(K;\Omega)  &\leq \int_{\Omega \setminus K} |\nabla \phi|^2 dV\\ 
&= \int_{\Omega \setminus K} f'(\psi)^2 |\nabla \psi|^2 dV\\ 
&=  \int_{0}^c f'(t)^2 \int_{\Sigma_t} |\nabla \psi| dA\, dt,
\end{align*}
where $\Sigma_t = \psi^{-1}(t)$. 
The idea is to choose $f$ to obtain the smallest possible value of this integral. Let $T(t) =  \int_{\Sigma_t} |\nabla \psi| dA$. 
From H\"older's inequality one can see the choice in which $f'$ is a constant multiple of $\frac{1}{T(t)}$ is what is sought. To obtain the correct boundary conditions, we choose
$$f(t) = \left(\int_0^c \frac{1}{T(\tau)}d\tau\right)^{-1} \int_0^t  \frac{1}{T(\tau)}d\tau.$$
With this choice we obtain the upper bound
$$\capac_g(K;\Omega) \leq \frac{1}{(n-2)\omega_{n-1}} \left(\int_0^c \frac{1}{\int_{\Sigma_t} |\nabla \psi| dA} \right)^{-1},$$
which only depends upon the choice of $\psi$.

Now consider the special case in which $\partial \Omega$ is a level set of the distance function from $K$, i.e. $\partial \Omega$ and $\partial K$ are a constant distance $c$ apart. Then we may use this distance function as $\psi$ above so $|\nabla \psi|=1$ and immediately obtain
\begin{equation}
\label{eqn_Sz}
\capac_g(K;\Omega) \leq \frac{1}{(n-2)\omega_{n-1}}  \left(\int_0^c \frac{1}{|\Sigma_t|} dt\right)^{-1},
\end{equation}
where $\Sigma_t$ is the surface whose distance from $K$ equals $t$ and $|\Sigma_t|$ is its area. This type of inequality was originally due to Szeg\"o; see \cite[Section 3.4]{PS}. 

Now we will apply this to the case in which $K$ and $\Omega$ are concentric small geodesic balls. Fix a point $p \in M$, and assume $0<R_1 < R_2$, where $R_2$ is less than the injectivity radius at $p$. We consider the balls $B_g(p,R_1) \subset B_g(p, R_2)$ about $p$.
By \eqref{eqn_Sz}:
\begin{equation}
\label{R1_R2_ub}
\capac_g(R_1, R_2)  \leq \frac{1}{(n-2) \omega_{n-1}} \left( \int_{R_1}^{R_2} \frac{1}{A(r)} dr\right)^{-1},
\end{equation}
where $A(r)$ is the area of the geodesic sphere of radius $r$. From \eqref{A},
\begin{equation}
A(r) = \omega_{n-1} r^{n-1}\left(1-\frac{S(p)}{6n}r^2\right) + E(r),
\end{equation}
where $|E(r)| \leq Kr^{n+3}$ for some constant $K$. It is then straightforward to estimate the following integral:
\begin{equation}
\omega_{n-1} \int_{R_1}^{R_2} \frac{1}{A(r)} dr = \int_{R_1}^{R_2} \frac{1}{ r^{n-1} \left(1-\frac{S(p)}{6n}r^2\right)} - \frac{E(r)}{\left[ r^{n-1}\left(1-\frac{S(p)}{6n}r^2\right)\right]\left[r^{n-1} \left(1-\frac{S(p)}{6n}r^2\right) + E(r)\right]}dr.\label{eqn_Er}
\end{equation}
The second summand in the integrand is at most $Kr^{5-n}$ for a possibly different constant $K$. The first term can be integrated and estimated explicitly. Let $\sigma = \frac{S(p)}{6n}$. Altogether, we obtain:
\begin{equation}
\label{integral_Ar}\omega_{n-1} \int \frac{1}{ A(r)} dr = 
\begin{cases}
-\frac{1}{r} + \sigma r + O(r^3), & n=3\\
-\frac{1}{2r^2} + \sigma \log(r) + O(r^2), & n=4\\
-\frac{1}{(n-2)r^{n-2}} -\frac{\sigma}{(n-4)r^{n-4}} +O(r^{6-n}), & n\geq 5, n \neq 6\\
-\frac{1}{4r^4} -\frac{\sigma}{2r^{2}} +\sigma^2 \log(r) + O(r^2), & n=6.
\end{cases}
\end{equation}

Now assume $R_2/R_1$ is a fixed constant $\lambda >1$ and let $r=R_1$. From \eqref{R1_R2_ub}, \eqref{eqn_Er}, and \eqref{integral_Ar}, it is elementary to obtain Proposition \ref{prop_ub}.
\end{proof}

\section{Lower bound on capacity}
\label{sec_lb}
In this section we prove a lower bound on the relative capacity of concentric geodesic balls:

\begin{prop}
\label{prop_lb}
Let $(M,g)$ be a Riemannian manifold of dimension $n \geq 3$, and let $p \in M$. Fix $\lambda > 1$. Then for $r>0$ small, the capacity of $\overline{B_g}(p,r)$ relative to $B_g(p, \lambda r)$ satisfies:
$$\capac_g(r, \lambda r) \geq 
\begin{cases}
\frac{r}{1-\lambda^{-1}}\left(1- \frac{S(p) }{18} \lambda r^2 + o(r^2)\right), & n=3\\
\frac{r^2}{1-\lambda^{-2}} \left(1-\frac{S(p)}{12}\frac{\log(\lambda)}{(1-\lambda^{-2})}r^2+o(r^2) \right), &n=4\\
\frac{r^{n-2}}{1-\lambda^{2-n}}\left(1-\frac{(n-2)S(p)}{6n(n-4)} \frac{1-\lambda^{4-n}}{1-\lambda^{2-n}}r^2 + o(r^2)\right), &n\geq 5.
\end{cases}$$
\end{prop}

Our approach is motivated by the proof of the capacity-volume inequality of Poincar\'e--Faber--Szeg\"o for Euclidean space \cite{PS} (also known as the isocapacitary inequality). Such an inequality (specifically, its relative version) need not hold in a Riemannian manifold, but the point is that on a small scale it will nearly hold; our goal is to extract the first nontrivial error term. The approach is also inspired by the complementary result on the capacity of large balls in an asymptotically flat manifold in \cite{Jau}.

We first recall an isoperimetric inequality that will be used in the proof. To motivate it, note formulas \eqref{V} and \eqref{A} together give control on the isoperimetric ratios of small geodesic balls, but in general balls are not isoperimetric minimizers, even locally. The following result of Druet provides an isoperimetric inequality for arbitrary regions inside a small ball.

\begin{thm}[Druet \cite{Dru}, equation (2.1)]
\label{thm_druet}
Let $(M,g)$ be a Riemannian manifold of dimension $n \geq 2$, and let $p \in M$. For any $\epsilon > 0$, there exists $r_\epsilon>0$ such that for any Borel set $\Omega \subseteq \overline{B_g}(p,r_\epsilon)$, 
$$|\partial^* \Omega|^2 \geq n^2 (\beta_n)^{\frac{2}{n}} |\Omega|^{\frac{2(n-1)}{n}} - \left(\frac{n}{n+2}S(p) + \epsilon \right) |\Omega|^2.$$
\end{thm}
Above, $|\partial^* \Omega|$ is the perimeter of $\Omega$ with respect to $g$ (which equals the boundary area of $\Omega$ if $\partial \Omega$ is sufficiently regular, e.g., $C^1$), and $|\Omega|$ is the volume of $\Omega$ with respect to $g$.

One important technical detail in the proof of Proposition \ref{prop_lb} pertains to the asymptotic behavior of the harmonic function that vanishes at radius $r$ and equals 1 at radius $\lambda r$, in the limit $r \to 0$. In light of the following lemma, these can be thought of as ``nearly radial'' harmonic functions:

\begin{lemma}
 \label{lemma_grad_u}
 Let $p$ be a point in a Riemannian $n$-manifold $(M,g)$, where $n \geq 3$. Fix a real number $\lambda>1$. For $r>0$ small, let $W_r$ be the closed metric annulus $\overline {B_g}(p,\lambda r) \setminus B_g(p,r)$. Let $u_r$ be the unique harmonic function on $W_r$ that vanishes on the inner boundary $\partial B_g(p,r)$ and equals 1 on the outer boundary $\partial B_g(p,\lambda r)$. Then in normal coordinates $y$ about $p$, for points in $W_r$,
 \begin{align}
u_r(y) &= \frac{1}{1-\lambda^{2-n}} \left(1-\frac{r^{n-2}}{|y|^{n-2}}\right) + O(r^{2}) \label{eqn_u}\\
 |\nabla u_r| &= \frac{(n-2)r^{n-2}}{1-\lambda^{2-n}} |y|^{1-n} + O(r)\label{eqn_Du}\\
\Hess(u_r)_{ij} &=  \frac{(n-2) r^{n-2}}{1-\lambda^{2-n}} \left( |y|^{-n} \delta_{ij} - \frac{n y^i y^j}{|y|^{n+2}}\right) + O(1),\label{eqn_DDu}
 \end{align}
 where the gradient norm and Hessian are taken with respect to $g$.
 \end{lemma}

We defer the proof until later in the section. We proceed with:

\begin{proof}[Proof of Proposition \ref{prop_lb}]
The initial part of the proof, through \eqref{eqn_I}, will be general (not specializing to balls) and will start out by following the standard proof of the capacity-volume inequality in the Euclidean case, as in \cite{PS}. Suppose $(M,g)$ is a Riemannian $n$-manifold, $n \geq 3$, with $K \subset \Omega \subset M$, where $K$ is compact and $\Omega$ is open with compact closure. Assume further that $\partial K$ and $\partial \Omega$ are smooth and nonempty. 

Let $u$ be the (harmonic) function realizing the infimum in $\capac_g(K; \Omega)$, and let $\Sigma_t$ be the level set $\{u=t\}$ for $t \in [0,1]$, which is smooth for almost all $t$. Then by the co-area formula and H\"older's inequality,
\begin{align*}
\capac_g(K; \Omega) &= \frac{1}{(n-2)\omega_{n-1}} \int_{\Omega \setminus K} |\nabla u|^2 dV\\
&= \frac{1}{(n-2)\omega_{n-1}} \int_0^1 \int_{\Sigma_t} |\nabla u| dA dt\\
& \geq \frac{1}{(n-2)\omega_{n-1}} \int_0^1 \frac{|\Sigma_t|^2}{\int_{\Sigma_t} \frac{1}{|\nabla u|} dA} dt.
\end{align*}
Let $\Omega_t = K \cup u^{-1}[0,t]$, so $\Omega_0 =K$ and $\Omega_1=\overline{\Omega}$. By the co-area formula,
$$|\Omega_t| = |K| + \int_0^t \int_{\Sigma_s} \frac{1}{|\nabla u|} dA ds,$$
so that for almost all $t$,
\begin{equation}
\label{eqn_d_dt_volume}
\frac{d}{dt}|\Omega_t| =  \int_{\Sigma_t} \frac{1}{|\nabla u|} dA.
\end{equation}
Define the isoperimetric ratio $I(t)$ of $\Omega_t$ by
$$\left(\frac{1}{\omega_{n-1}} |\Sigma_t|\right)^\frac{n}{n-1} = \frac{1}{\beta_n} |\Omega_t| \;I(t),$$
which is finite for almost all $t$. 
(If $(M,g)$ is Euclidean, the isoperimetric inequality states that $I(t) \geq 1$.) 
Inserting \eqref{eqn_d_dt_volume} and using the definition of $I(t)$, we have
\begin{align}
\capac_g(K; \Omega) &\geq\frac{\omega_{n-1}}{(n-2)(\beta_n)^{\frac{2(n-1)}{n}}} \int_0^1 \frac{|\Omega_t|^{\frac{2(n-1)}{n}} I(t)^{\frac{2(n-1)}{n} } }{\int_{\Sigma_t} \frac{1}{|\nabla u|} dA} dt\nonumber\\
&= \underbrace{\frac{\omega_{n-1}}{(n-2)(\beta_n)^{\frac{2(n-1)}{n}}} \int_0^1 \frac{|\Omega_t|^{\frac{2(n-1)}{n}} }{\frac{d}{dt}|\Omega_t|} dt}_{\I}+ \underbrace{\frac{\omega_{n-1}}{(n-2)(\beta_n)^{\frac{2(n-1)}{n}}} \int_0^1 \frac{|\Omega_t|^{\frac{2(n-1)}{n}} \left(I(t)^{\frac{2(n-1)}{n} }- 1\right) }{\int_{\Sigma_t} \frac{1}{|\nabla u|} dA} dt}_{\II}. \label{eqn_I_II}
\end{align}

We will estimate terms $\I$ and $\II$ separately. For $\I$, we continue to follow \cite{PS} and let $R(t)$ be the  volume radius of $\Omega_t$, i.e.
$$\beta_n R(t)^n = |\Omega_t|,$$
so for almost all $t$,
$$\frac{d}{dt}|\Omega_t| = \omega_{n-1} R(t)^{n-1} R'(t).$$ Then $\I$ becomes
$$\I = \frac{1}{(n-2)} \int_0^1 \frac{R(t)^{n-1} }{R'(t)} dt.$$

Now, let $\hat K$ and $\hat \Omega$ be concentric balls in Euclidean $n$-space (closed and open, respectively) with volumes equal to $|K|$ and $|\Omega|$, respectively. Let $\hat \Sigma_t$ be the Euclidean sphere with the same center as $\hat K$ and $\hat \Omega$ enclosing volume equal to $|\Omega_t|$. Then $\hat \Sigma_0 = \partial \hat K$ and $\hat \Sigma_1 = \partial \hat \Omega$. Let $\hat u$ be the function that equals $t$ on $\hat \Sigma_t$. For almost all $t$,
$$|\nabla \hat u| = \frac{1}{R'(t)}$$
on $\Sigma_t$. Then continuing the above and using the co-area formula:
\begin{align*}
\I &= \frac{1}{(n-2)\omega_{n-1}} \int_0^1 \int_{\hat \Sigma_t} |\nabla \hat u| dA dt\\
&= \frac{1}{(n-2)\omega_{n-1}} \int_{\hat \Omega \setminus \hat K} |\nabla \hat u|^2 dV\\
&\geq \capac_0(\hat K; \hat \Omega),
\end{align*}
by definition, the latter being the relative capacity in Euclidean space. (While it is not obvious that $\hat u$ is Lipschitz, the above calculation shows it is in $W^{1,2}$, and a standard  smoothing argument implies the desired inequality.) This value is readily calculated in terms of the volumes of $\hat \Omega$ and $\hat K$ using \eqref{eqn_euc_cap}, and hence in terms of $|K|$ and $|\Omega|$. Specifically, we obtain:
\begin{equation}
\label{eqn_I}
\I \geq \left(\left(\frac{\beta_n}{|K|}\right)^{\frac{n-2}{n}}-\left(\frac{\beta_n}{|\Omega|}\right)^{\frac{n-2}{n}}\right)^{-1}.
\end{equation}
This is where the ``classical'' part of the proof concludes. 

Now we specialize to the case in which $K$ and $\Omega$ are concentric small geodesic balls about some $p \in M$, with fixed ratio of their radii. Take $K = \overline {B_g}(p,r)$ and $\Omega = B_g(p, \lambda r)$ for a constant $\lambda>1$. We assume $2\lambda r$ is less than the injectivity radius at $p$, so that all metric spheres of radius less than $2\lambda r$ are smooth. The harmonic function previously called $u$, now vanishing on $\partial B_g(p,r)$ and equaling 1 on $\partial B_g(p,\lambda r)$, will be denoted by $u_r$.

We note from the volume expansion \eqref{V} that for $r$ small,
$$\left(\frac{\beta_n}{|B_g(p,r)|}\right)^{\frac{n-2}{n}} = \frac{1}{r^{n-2}} \left(1 + \frac{(n-2)S(p)}{6n(n+2)}r^2 + O(r^4) \right).$$
Using this at radii $r$ and $\lambda r$, \eqref{eqn_I} produces:
\begin{equation}
\label{eqn_I_lb}
\I \geq \frac{r^{n-2}}{1- \lambda^{2-n}} \left(1- \frac{(n-2)S(p)}{6n(n+2)}\cdot \frac{1-\lambda^{4-n}}{1-\lambda^{2-n}} r^2 + O(r^4) \right).
\end{equation}
Note the leading order term is $c_n(r,\lambda r)$, the Euclidean relative capacity.

We move on to term $\II$ in \eqref{eqn_I_II}. We first need to control the location of the level sets of $u_r$, as this will lead to volume and area bounds. From Lemma \ref{lemma_grad_u}, it follows that in a normal coordinate system $y$ about $p$,
\begin{equation}
\label{eqn_u_r}
u_r(y) = \frac{1}{1-\lambda^{2-n}}\left(1-\frac{r^{n-2}}{|y|^{n-2}}\right) + F_r(y),
\end{equation}
on the annular region $W_r=\overline {B_g}(p,\lambda r) \setminus B_g(p,r)$ for a smooth function $F_r$, where $|F_r(y)| \leq Cr^2$ for a constant $C$ independent of $r$.
We continue to let $\Sigma_t$ denote the $t$-level set of $u_r$, omitting $r$ from the notation.
If $y\in \Sigma_t$, then from \eqref{eqn_u_r} it follows that
\begin{equation}
\label{eqn_level_set}
\rho_-(t) \leq |y| \leq \rho_+(t), 
\end{equation}
where we introduce
\begin{align*}
\rho_-(t) &= \frac{r}{\left(1 - t(1-\lambda^{2-n}) + C' r^2 \right)^{\frac{1}{n-2}}},\\
\rho_+(t) &= \frac{r}{\left(1 - t(1-\lambda^{2-n}) - C' r^2 \right)^{\frac{1}{n-2}}},
\end{align*}
for a constant $C'$ depending only on $C$ and $\lambda$.
(We can shrink $r$ if necessary to make the denominator in $\rho_+(t)$ positive for all $t \in [0,1]$.) We note for later reference that
\begin{align}
\rho_-(t) &\geq r + O(r^3) \label{eqn_r_minus}\\
\rho_+(t) &\leq \lambda r + O(r^3). \label{eqn_r_plus}
\end{align}
Now, from \eqref{eqn_level_set}, $\Sigma_t$ encloses the sphere $\partial B_g(p,\rho_-(t))$ and is enclosed by the sphere $\partial B_g(p,\rho_+(t))$. (Again, we shrink $r$ if necessary to arrange $\rho_+(t) \leq 2\lambda r$ to guarantee smoothness of these spheres.)
Thus, 
\begin{equation}
\label{eqn_balls}
B_g(p,\rho_-(t)) \subseteq \Omega_t\subseteq B_g(p,\rho_+(t)),
\end{equation}
where we recall $\Omega_t=K \cup u_r^{-1}[0,t]$ is the compact region bounded by $\Sigma_t$. This leads to an immediate volume comparison:
\begin{equation}
\label{vol_compare}
V(\rho_-(t)) \leq |\Omega_t| \leq V(\rho_+(t)).
\end{equation}
A corresponding area comparison holds as well, but it is not as obvious:
\begin{lemma}
\label{lemma_area}
For all $r>0$ sufficiently small, we have
$$A(\rho_-(t)) \leq |\Sigma_{t}| \leq A(\rho_+(t)).$$
\end{lemma}
This will be proved later in the section.

 Now we work on estimating $\II$. Recall we have:
\begin{align*}
-\II &= \frac{\omega_{n-1}}{(n-2)(\beta_n)^{\frac{2(n-1)}{n}}} \int_0^1 \frac{|\Omega_t|^{\frac{2(n-1)}{n}} \left(1-I(t)^{\frac{2(n-1)}{n} }\right) }{\int_{\Sigma_t} \frac{1}{|\nabla u_r|} dA} dt.
\end{align*}
We first address the isoperimetric ratio term. Let $\epsilon > 0$ be given. Applying Theorem \ref{thm_druet} to $\Omega_t \subseteq \overline{B_g}(p, \lambda r)$, we have (for $r>0$ sufficiently small and almost all $t$):
\begin{align*}
1-I(t)^{\frac{2(n-1)}{n}} &\leq (\gamma+\epsilon) |\Omega_t|^{2/n},
\end{align*}
where
$$\gamma=\frac{n^{2/n}}{(\omega_{n-1})^{2/n} } \cdot \frac{S(p)}{n(n+2)}=   \frac{S(p)}{n(n+2)\beta_n^{2/n}}.$$

\noindent \paragraph{\emph{Case 1:}} We assume $S(p) \geq 0$, so that $\gamma + \epsilon > 0$.

To continue, we need an estimate for the gradient of $u_r$ on $W_r$, provided by Lemma \ref{lemma_grad_u}:
\begin{align}
|\nabla u_r|  &= (n-2)c_n(r,\lambda r) |y|^{1-n}+ O(r) \label{grad_u_A0}
 \end{align}
(Below we will denote $c_n(r,\lambda r)$ by $c_n$ to economize on notation, keeping in mind $c_n$ is $O(r^{n-2})$.) Then using Lemma \ref{lemma_area} and \eqref{eqn_level_set},
\begin{align}
\int_{\Sigma_t} \frac{1}{|\nabla u_r|} dA &\geq |\Sigma_t| \left((n-2)c_n \rho_-(t)^{1-n} + O(r)\right)^{-1} \nonumber\\
&\geq A(\rho_-(t)) \left((n-2)^{-1}  c_n^{-1} \rho_-(t)^{n-1} + O(r^{3})\right). \label{grad_u_A}
\end{align}

Using \eqref{vol_compare} and \eqref{grad_u_A}, followed by the volume \eqref{V} and area \eqref{A} expansions for small metric balls:
\begin{align}
-\II &\leq \frac{\omega_{n-1}}{(n-2)(\beta_n)^{\frac{2(n-1)}{n}}} \int_0^1 \frac{|\Omega_t|^{2}(\gamma+\epsilon) }{\int_{\Sigma_t} \frac{1}{|\nabla u_r|} dA} dt \label{II_start}\\
     &\leq \frac{\omega_{n-1}}{(n-2)(\beta_n)^{\frac{2(n-1)}{n}}} \int_0^1 \frac{V(\rho_+(t))^{2}(\gamma+\epsilon) }{ A(\rho_-(t)) \left((n-2)^{-1}  c_n^{-1} \rho_-(t)^{n-1} + O(r^{3})\right) } dt \nonumber\\
     &\leq (\gamma+\epsilon)(\beta_n)^{\frac{2}{n}} c_n \int_0^1 \frac{\rho_+(t)^{2n}\left( 1 - \frac{S(p)}{3(n+2)}\rho_+(t)^2+ O(\rho_+(t)^{n+4})\right) }{ \rho_-(t)^{n-1}\left(1-\frac{S(p)}{6n}\rho_-(t)^2 + O(\rho_-(t)^4)\right) \left( \rho_-(t)^{n-1} + O(r^{n+1}) \right) } dt.\nonumber
\end{align}
Next, thanks to \eqref{eqn_r_minus} and \eqref{eqn_r_plus}, we can replace $\rho_{\pm }(t)$ in the error terms with $r$:
  \begin{align*}
 -\II &\leq (\gamma+\epsilon)(\beta_n)^{\frac{2}{n}} c_n \int_0^1 \frac{\rho_+(t)^{2n}\left( 1 + O(r^2)\right)}{ \rho_-(t)^{2n-2}\left(1+O(r^2)\right) \left( 1 + O(r^{2})\right) } dt.
 \end{align*}
To continue, we need to control the ratio $\rho_+(t) / \rho_-(t)$. We have
\begin{align*}
\frac{\rho_+(t)}{\rho_-(t)} &= \left(\frac{1 - t(1-\lambda^{2-n}) + C' r^2}{1 - t(1-\lambda^{2-n}) - C' r^2}\right)^{\frac{1}{n-2}}.
\end{align*}
This is maximized at $t=1$, and consequently we find:
\begin{align*}
\frac{\rho_+(t)}{\rho_-(t)} &\leq 1+O(r^2).
\end{align*}
Then we have:
\begin{align}
\II &\geq -c_n \left(\frac{S(p)}{n(n+2)}+\epsilon(\beta_n)^{\frac{2}{n}} )\right)(1+O(r^2)) \int_0^1 \rho_+(t)^{2} dt. \label{eqn_II_final}
\end{align}
The $\rho_+(t)^2$ integral is elementary to evaluate:
\begin{equation}
\label{eqn_integral}\int_0^1 \rho_+(t)^2 dt = \begin{cases}
\lambda r^2 + O(r^3),& n=3\\
\frac{2 \log(\lambda) }{1-\lambda^{-2}}r^2+O(r^3),& n=4\\
\frac{n-2}{n-4}\cdot \frac{1-\lambda^{4-n}}{1-\lambda^{2-n}}r^2 + O(r^3), &n=5.
\end{cases}
\end{equation}
In our use of Theorem \ref{thm_druet}, $\epsilon$ can be chosen to decrease to 0 as $r \to 0$, though we do not control the rate; thus, the $\epsilon$ term can be regarded as $o(1)$ in $r$.

Finally, from \eqref{eqn_I_II}, \eqref{eqn_I_lb},  \eqref{eqn_II_final} and \eqref{eqn_integral}, we have (for $n\geq 5$)
\begin{align}
\capac_g(r,\lambda r) &\geq c_n(r,\lambda r)\left(1- \frac{(n-2)S(p)}{6n(n+2)}\cdot \frac{1-\lambda^{4-n}}{1-\lambda^{2-n}} r^2 - \frac{S(p)}{n(n+2)}\cdot \frac{n-2}{n-4}\cdot \frac{1-\lambda^{4-n}}{1-\lambda^{2-n}}r^2 + o(r^2) \right) \nonumber \\
 &\geq c_n(r,\lambda r)\left(1- \frac{(n-2)S(p)}{6n(n-4)}\cdot \frac{1-\lambda^{4-n}}{1-\lambda^{2-n}} r^2 + o(r^2) \right). \label{eqn_cap_r_rho_r}
 \end{align}
The $n=3,4$ cases are very similar. This proves Proposition \ref{prop_lb} for the case $S(p)\geq 0$.

\medskip

\noindent \paragraph{\emph{Case 2:}} Now assume instead that $S(p) < 0$, i.e., $\gamma < 0$. We shrink $\epsilon > 0$ if necessary to arrange $\gamma+\epsilon < 0$. Then the above argument can be modified to still obtain \eqref{eqn_cap_r_rho_r} (and the corresponding inequalities for $n=3,4$) as follows. Beginning at \eqref{II_start}, use the bound $|\Omega_t| \geq V(\rho_-(t))$ from \eqref{vol_compare}. To estimate the $|\nabla u_r |$ term as in \eqref{grad_u_A} we instead use 
$$\int_{\Sigma_t} \frac{1}{|\nabla u_r|} dA \leq A(\rho_+(t)) \left((n-2)^{-1}  c_n^{-1} \rho_+(t)^{n-1} + O(r^{3})\right),$$
which is justified by \eqref{grad_u_A0}, \eqref{eqn_level_set}, and Lemma \ref{lemma_area}. The remainder of the argument is entirely analogous. For example, a lower bound on $\rho_-(t) / \rho_+(t)$ is $1+O(r^2)$, and the evaluation of $\int_0^1 \rho_-(t)^2 dt$ differs from $\int_0^1 \rho_+(t)^2 dt$ only by  error terms of the same order as in \eqref{eqn_integral}.

Aside from the proofs of Lemmas  \ref{lemma_grad_u} and  \ref{lemma_area}, Proposition \ref{prop_lb} follows.
\end{proof}

\begin{proof}[Proof of Lemma \ref{lemma_grad_u}]
Fix a constant $\lambda >1$. Fix a normal coordinate chart $y$ about $p$.

Let $\Omega  \subset \R^n$ be the closed annular region given by $1 \leq |x| \leq \lambda$, where $x$ denotes the standard coordinates. Let $\varphi_0$ be the harmonic function with respect to the Euclidean metric on $\Omega$ that vanishes on the inner boundary and equals one on the outer boundary. Explicitly,
$$\varphi_0(x) = \frac{1}{1-\lambda^{2-n}} \left(1- \frac{1}{|x|^{n-2}}\right).$$

Let $\Phi_r: \Omega \to W_r$ be the diffeomorphism given in the chart $y$ by $\Phi_r(x) = r x$.

Let $u_r$ be $g$-harmonic on $W_r$, vanishing on the inner boundary at radius $r$ and equaling 1 on the outer boundary at radius $\lambda r$. We ``blow up'' $u_r$ by defining the function $\varphi_r$ on $\Omega$ by:
$$\varphi_r(x) = u_r \circ \Phi_r(x).$$
Note that $\varphi_r$ is harmonic with respect to the following Riemannian metric on $\Omega$: $$\hat g_r = r^{-2} \Phi_r^* g.$$
We apply the global Schauder inequality \cite[Theorem 6.6]{GT}  on $\Omega$ to $\varphi_r -\varphi_0$ and the elliptic operator $\Delta_{\hat g_r}$. As is well known, a maximum principle argument and the fact that $\varphi_r - \varphi_0$ vanishes on $\partial \Omega$ allows one to drop the $C^0$ norm of $\varphi_r -\varphi_0$ on the right-hand side of the inequality. Specifically, we have:
$$\|\varphi_r - \varphi_0\|_{C^{2,\alpha}(\Omega)} \leq C \|\Delta_{\hat g_r} (\varphi_r - \varphi_0)\|_{C^{0,\alpha}(\Omega)}=C \|\Delta_{\hat g_r} \varphi_0\|_{C^{0,\alpha}(\Omega)},$$
where the H\"older norms are taken in the coordinate chart $x$. Since the family $\hat g_r$ converges uniformly to the Euclidean metric on $\Omega$ as $r \to 0$, the constant $C$ may be taken independent of $r$. From scaling and diffeomorphism invariance, 
$$\|\Delta_{\hat g_r} \varphi_0\|_{C^{0,\alpha}(\Omega)} =  r^2 \| \Phi_r^* (\Delta_{g} u_0)\|_{C^{0,\alpha}(\Omega)},$$ 
where $u_0 = \varphi_0 \circ \Phi_r^{-1}$, i.e. 
$$u_0(y) = \frac{1}{1-\lambda^{2-n}} \left(1- \frac{r^{n-2}}{|y|^{n-2}}\right).$$
The $C^0$ part of the H\" older norm is independent of the diffeomorphism, but an $r^{\alpha}$ factor appears in semi-norm part. That is,
\begin{equation}
\label{eqn_seminorm}
\|\Phi_r^* \Delta_g u_0\|_{C^{0,\alpha}(\Omega)} =  \|\Delta_{g} u_0\|_{C^{0}(W_r)} + r^{\alpha} \sup_{y \neq y'\\ \text{ in } W_r} \frac{|\Delta_g u_0(y) - \Delta_g u_0(y')|}{|y-y'|^\alpha}
\end{equation}

Using the fact that $y$ represents normal coordinates, it is straightforward to check that $\Delta_{g} u_0$ is $r^{n-2}$ times a function $Q(y)$ that is smooth away from 
$p$ and is independent of the parameter $r$, where $Q(y) = O(|y|^{2-n})$ 
and $\frac{\partial Q}{\partial y^i} = O(|y|^{1-n})$, etc. ($Q(y)$ depends on the metric coefficients, the function $|y|^{2-n}$, and their derivatives.)
Specifically, the semi-norm part of $\|\Delta_{g} u_0\|_{C^{0,\alpha}(\Omega)}$ is $O(|y|^{2-n-\alpha})$. Then from \eqref{eqn_seminorm},
$$\|\Phi_r^* \Delta_g u_0\|_{C^{0,\alpha}(\Omega)} = O(1),$$
since $|y| = O(r)$ on $W_r$.
Putting the last few equations together,
$$\|\varphi_r - \varphi_0\|_{C^{2,\alpha}(\Omega)} \leq C r^{2},$$
for a possibly different constant $C$ that is independent of $r$,
i.e.
\begin{equation}
\label{u_r_u_0}
\| \Phi_r^* (u_r - u_0)\|_{C^{2,\alpha}(\Omega)} \leq C r^{2}.
\end{equation}

We look at the consequences of \eqref{u_r_u_0}. First we have a $C^0$ estimate on $u_r$:
\begin{align*}
u_r(y) &= \frac{1}{1-\lambda^{2-n}} \left(1-  \frac{r^{n-2}}{|y|^{n-2}}\right) + O(r^{2})
\end{align*}
which proves \eqref{eqn_u}.

From \eqref{u_r_u_0}, the gradient estimate follows
$$\sup_{y \in W_r} |\nabla u_r(y) - \nabla u_0(y)| \leq C r,$$
where the gradient and norm are with respect to $g$. Using the fact that $y$ is a normal coordinate system, we can relate the $g$-gradient norm of $u_0$ back to its norm in the coordinate chart, obtaining \eqref{eqn_Du}:
$$|\nabla u_r| =  \frac{(n-2) r^{n-2}}{1-\lambda^{2-n}} \frac{1}{|y|^{n-1}} +O(r).$$

Finally, from \eqref{u_r_u_0}, we have
$$\frac{\partial^2 u_r}{\partial y^i\partial y^j} = \frac{\partial^2 u_0}{\partial y^i\partial y^j} + O(1).$$
Computing the first term on the right-hand side explicitly and using the fact that $y$ is a normal coordinate system (so that the Christoffel symbols are $O(|y|)$, it follows
$$\Hess(u_r)_{ij} =  \frac{(n-2) r^{n-2}}{1-\lambda^{2-n}} \left( |y|^{-n} \delta_{ij} - \frac{n y^i y^j}{|y|^{n+2}}\right) + O(1),$$
where the Hessian is taken with respect to $g$. This gives \eqref{eqn_DDu}.
\end{proof}

\begin{proof}[Proof of Lemma \ref{lemma_area}]
We begin with the following general fact: let $\{N_s\}_{a \leq s \leq b}$ be a smooth foliation of a region $W$ in a Riemannian manifold, where the $N_s$ are compact hypersurfaces that are convex in the direction of increasing $s$. Then any hypersurface in $W$ (not necessarily a leaf of the foliation) homologous to $N_a$ within $W$ has area at least as large as $|N_a|$. This can be seen by noting the map that collapses $W$ onto $N_a$ along normals of the foliation is distance non-increasing and hence area non-increasing.

Next, we recall the setup and notation in Lemma \ref{lemma_area}. We have the harmonic function $u_r$ on the annular region $W_r$ about $p$, whose level sets are denoted $\Sigma_t$ for $0 \leq t \leq 1$. With respect to normal coordinates $y$ about $p$, each point in $\Sigma_t$ has distance at least $\rho_-(t)$ and at most $\rho_+(t)$ from $p$.

Note all sufficiently small metric spheres about $p$ are smooth and convex and foliate a punctured neighborhood of $p$. Since $\Sigma_t$ encloses $\partial B_g(p,\rho_-(t))$ by \eqref{eqn_balls}, it follows from the above fact that (if $r>0$ is sufficiently small)
$$A(\rho_-(t)) \leq |\Sigma_t|,$$
which proves the first part of the lemma. 

We will also need to apply the above fact to the foliation of $W_r$ by the level sets of $u_r$. Note from \eqref{eqn_Du} in Lemma \ref{lemma_grad_u}, for $r$ sufficiently small, $|\nabla u_r|$ is nonzero on $W_r$; in particular, all level sets $\Sigma_{t}$ are smooth hypersurfaces that foliate $W_r$. We claim that for $r>0$ sufficiently small, these level sets all are convex.

Letting $B_t$ denote the (scalar-valued) second fundamental form of $\Sigma_t$ in the (outward-pointing) unit normal direction $\frac{\nabla u_r}{|\nabla u_r|}$, we have
$$B_t(\cdot, \cdot) = \frac{1}{|\nabla u_r|} \Hess (u_r)(\cdot, \cdot).$$
Direct computation using Lemma \ref{lemma_grad_u} then shows that $B_t$ is positive definite on directions tangential to $\Sigma_t$ for $r>0$ small; i.e., the $\Sigma_t$ are convex.

We now consider two cases to complete the proof.

Case 1: $\rho_+(t) \geq \lambda r$. Since $A(\cdot)$ is increasing for $r$ sufficiently small by \eqref{A}, we have $A(\rho_+(t)) \geq A(\lambda r)$. Next, since $\partial B_g(p,\lambda r)$ is the $t=1$ level set of $u_r$ and $\Sigma_t$ is a $t\leq 1$ level set, and these level sets are convex, we have $|\Sigma_t| \leq A(\lambda r)$, which completes the proof in this case.

Case 2: $\rho_+(t) < \lambda r$. Then $\partial B_g(p,\rho_+(t))$ lies in the region foliated by the level sets of $u$ between the values $t$ and $1$. Then by the fact at the beginning of the proof,
\begin{equation*}
|\Sigma_t| \leq |\partial B_g(p,\rho_+(t))| = A(\rho_+(t)).\qedhere
\end{equation*}
\end{proof}

Finally, Theorem \ref{thm_main} follows immediately from Propositions \ref{prop_ub} and \ref{prop_lb}.

\begin{remark}
In the case in which $(M,g)$ has constant scalar curvature and ``strong bounded geometry,'' it should be possible using the small-volume isoperimetric profile expansion in \cite{Nar} (in place of Theorem \ref{thm_druet}) to improve the $o(r^2)$ error estimate in Proposition \ref{prop_lb} and hence in Theorem \ref{thm_main} to $O(r^4)$. The constant (zero) scalar curvature case is relevant to Conjecture \ref{conj_cap}.
\end{remark}


\begin{bibdiv}
 \begin{biblist}
 
 \bib{Bray_RPI}{article}{
   author={Bray, H.},
   title={Proof of the Riemannian Penrose inequality using the positive mass
   theorem},
   journal={J. Differential Geom.},
   volume={59},
   date={2001},
   number={2},
   pages={177--267}
}

 
 \bib{BKKS}{article}{
   author={Bray, H.},
   author={Kazaras, D.},
   author={Khuri, M.},
   author={Stern, D.},
   title={Harmonic Functions and The Mass of 3-Dimensional Asymptotically Flat Riemannian Manifolds},
   eprint={https://arxiv.org/abs/1911.06754}
}

 
 \bib{BM}{article}{
   author={Bray, H.},
   author={Miao, P.},
   title={On the capacity of surfaces in manifolds with nonnegative scalar
   curvature},
   journal={Invent. Math.},
   volume={172},
   date={2008},
   number={3},
   pages={459--475}
}

\bib{Bur}{article}{
   author={Burkhardt-Guim, P.},
   title={Pointwise lower scalar curvature bounds for $C^0$ metrics via
   regularizing Ricci flow},
   journal={Geom. Funct. Anal.},
   volume={29},
   date={2019},
   number={6},
   pages={1703--1772}
}



\bib{Dru}{article}{
   author={Druet, O.},
   title={Sharp local isoperimetric inequalities involving the scalar
   curvature},
   journal={Proc. Amer. Math. Soc.},
   volume={130},
   date={2002},
   number={8},
   pages={2351--2361}
}

\bib{FST}{article}{
   author={Fan, X.-Q.},
   author={Shi, Y.},
   author={Tam, L.-F.},
   title={Large-sphere and small-sphere limits of the Brown-York mass},
   journal={Comm. Anal. Geom.},
   volume={17},
   date={2009},
   number={1},
   pages={37--72},
}


\bib{SY}{article}{
   author={Schoen, R.},
   author={Yau, S.-T.},
   title={Conformally flat manifolds, Kleinian groups and scalar curvature},
   journal={Invent. Math.},
   volume={92},
   date={1988},
   number={1},
   pages={47--71}
}


\bib{FS}{article}{
   author={Freire, A.},
   author={Schwartz, F.},
   title={Mass-capacity inequalities for conformally flat manifolds with
   boundary},
   journal={Comm. Partial Differential Equations},
   volume={39},
   date={2014},
   number={1},
   pages={98--119}
 }



\bib{GT}{book}{
   author={Gilbarg, D.},
   author={Trudinger, N.},
   title={Elliptic partial differential equations of second order},
   series={Classics in Mathematics},
   publisher={Springer-Verlag, Berlin},
   date={2001}
}	


\bib{GV}{article}{
   author={Gray, A.},
   author={Vanhecke, L.},
   title={Riemannian geometry as determined by the volumes of small geodesic
   balls},
   journal={Acta Math.},
   volume={142},
   date={1979},
   number={3-4},
   pages={157--198}
}

\bib{Gro}{article}{
   author={Gromov, M.},
   title={Dirac and Plateau billiards in domains with corners},
   journal={Cent. Eur. J. Math.},
   volume={12},
   date={2014},
   number={8},
   pages={1109--1156}
}

\bib{Hui}{article}{
     author={Huisken, G.},
     title={An isoperimetric concept for mass and quasilocal mass},
     journal={Oberwolfach Reports, European Mathematical Society (EMS), Z\"urich},
   date={2006},
   volume={3},
   number={1},
   pages={87--88}
}

\bib{Jau}{article}{
   author={Jauregui, J.},
   title={ADM mass and the capacity-volume deficit at infinity},
   eprint={https://arxiv.org/abs/2002.08941}
}

\bib{Lee}{article}{
   author={Lee, D.},
   title={A positive mass theorem for Lipschitz metrics with small singular
   sets},
   journal={Proc. Amer. Math. Soc.},
   volume={141},
   date={2013},
   number={11},
   pages={3997--4004}
}

\bib{CM}{article}{
   author={Li, C.},
   author={Mantoulidis, C.},
   title={Positive scalar curvature with skeleton singularities},
   journal={Math. Ann.},
   volume={374},
   date={2019},
   number={1-2},
   pages={99--131}
}


\bib{MMT}{article}{
   author={Mantoulidis, C.},
   author={Miao, P.},
   author={Tam, L.-F.},
   title={Capacity, quasi-local mass, and singular fill-ins},  
   journal={J. Reine Angew. Math. (to appear)}
}

\bib{Nar}{article}{
   author={Nardulli, S.},
   author={Osorio Acevedo, L. E.},
   title={Sharp Isoperimetric Inequalities for Small Volumes in Complete
   Noncompact Riemannian Manifolds of Bounded Geometry Involving the Scalar
   Curvature},
   journal={Int. Math. Res. Not. IMRN},
   date={2020},
   number={15},
   pages={4667--4720}
}

\bib{PS}{book}{
   author={P{\'o}lya, G.},
   author={Szeg{\"o}, G.},
   title={Isoperimetric Inequalities in Mathematical Physics},
   series={Annals of Mathematics Studies, no. 27},
   publisher={Princeton University Press},
   place={Princeton, N. J.},
   date={1951}
}

\bib{Schw}{article}{
   author={Schwartz, F.},
   title={A volumetric Penrose inequality for conformally flat manifolds},
   journal={Ann. Henri Poincar\'{e}},
   volume={12},
   date={2011},
   number={1},
   pages={67--76}
}

\bib{Ste}{article}{
   author={Stern, D.},
   title={Scalar curvature and harmonic maps to $S^1$},
   eprint={https://arxiv.org/abs/1908.09754}
}


\end{biblist}
\end{bibdiv}
\end{document}